\documentclass[11pt]{amsart}
\usepackage[leqno]{amsmath}
\usepackage{epsfig}
\usepackage{epsfig}
\usepackage{color}
\usepackage{amsmath}
\usepackage{amssymb}
\newtheorem{lemma}{Lemma}
\newtheorem{proposition}{Proposition}
\newtheorem{theorem}{Theorem}

\theoremstyle{definition}

\usepackage{color}
\newcommand{\commentout}[1]{}
\newcommand{\coll}{\;\;\makebox[0pt]{$\bot$}\makebox[0pt]{$\smile$}\;\;}
\usepackage{graphicx}

\begin{document}

\thispagestyle{empty}

\centerline{\Large\bf Nice labeling problem for event structures: a counterexample}

\bigskip

\hspace*{7cm}{\it For Hans-J\"urgen Bandelt on his 60th birthday}

%

\vspace{10mm}

\centerline{{\sc Victor Chepoi} }

\vspace{3mm}

\medskip
\centerline{Laboratoire d'Informatique Fondamentale,}
\centerline{Universit\'e d'Aix-Marseille,}
\centerline{Facult\'e des Sciences de Luminy,} \centerline{F-13288
Marseille Cedex 9, France} \centerline{chepoi@lif.univ-mrs.fr}

\vspace{15mm}
\begin{footnotesize} \noindent {\bf Abstract.} In this note, we present a counterexample to a conjecture
of Rozoy and Thiagarajan from 1991 (called also the nice labeling problem) asserting that any (coherent) event structure
with finite degree admits a labeling with a finite number of labels, or equivalently, that there exists a function
$f: \mathbb{N} \mapsto \mathbb{N}$ such that an event structure with degree $\le n$ admits a labeling with at most $f(n)$
labels. Our counterexample is based on the Burling's construction from 1965 of 3-dimensional box hypergraphs with clique number 2
and arbitrarily large chromatic numbers and the bijection between domains of event structures and median graphs established by
Barth\'elemy and Constantin in 1993.
\end{footnotesize}

\section{Introduction}

Event structures introduced by Nielsen, Plotkin, and Winskel \cite{NiPlWi,Wi,WiNi} is a widely recognized abstract model of concurrent computation. An event
structure is a partially ordered set of the occurrences of actions, called events, together with a conflict relation.
The partial order captures the causal dependency of events. The conflict relation models incompatibility of events so that two events that
are in conflict cannot simultaneously occur  in any state of the computation. Consequently, two events that are neither ordered nor in conflict may
occur concurrently.  Formally, an {\it event structure} is a triple ${\mathcal E}=(E,\le, \smile),$ where

\begin{itemize}
\item{} $E$ is a set of {\it events},
\item{} $\le\subseteq E\times E$ is a partial order of {\it causal dependency},
\item{} $\smile\subseteq E\times E$ is a binary, irreflexive, symmetric relation of {\it conflict},
\item{} $e\smile e'$ and $e'\le e''$ imply $e\smile e''$.
\end{itemize}

\noindent
What we call here an event structure is usually called a coherent event structure or an event structure with a binary conflict. Additionally, the partial order $\le$ in the definition of an event structure is supposed to be {\it finitary,}
i.e., the set $\{ e'\in E: e'\le e\}$ is finite for any $e\in E.$ Two  events $e',e''$ are {\it concurrent} (notation $e'\frown e''$) if they are order-incomparable and they are not in conflict.
Let $e'$ and $e''$ be two elements in conflict. This conflict $e'\smile e''$ is said to be {\it minimal} if there is no element $e\ne e',e''$ such that either $e\le e'$ and $e\smile e''$ or $e\le e''$ and $e\smile e'$. Two elements are {\it independent} \cite{AsBouChaRo} (or {\it orthogonal} \cite{Sa}) if they are either concurrent or in minimal conflict. An {\it independent} set is a subset of $E$ whose elements are pairwise independent. The {\it degree} of an event structure ${\mathcal E}$ is the least upper bound of the sizes of the independent sets.

A labeling of an event structure ${\mathcal E}$ is a map $\lambda$ from $E$ to some alphabet $\Lambda$. The labeling $\lambda$ is a {\it nice labeling} of  $\mathcal E$ if any two independent events have different labels.  Assous, Bouchitt\'e, Charretton, and Rozoy \cite{AsBouChaRo} note that a nice labeling of an event structure ``is equivalent to label the transitions by actions with the following condition: two transitions associated with the same initial state but with two different final states must have two different labels" and that the nice labeling conjecture of \cite{RoTh} formulated below arises when studying the equivalence of three different models of distributed computation: labeled event structures, transitions systems, and distributed monoids.
Nice labeling of event structures was introduced by Rozoy and Thiagarajan \cite{RoTh} in their study of relationships between trace monoids and labeled event structures.
Rozoy and Thiagarajan conjectured that {\it any event structure with finite degree admits a nice labeling with a finite number of labels}.  In a quantitative version
of this conjecture (which can be  considered also for finite event structures) this conjecture can be re-formulated as: {\it there exists a function
$f: \mathbb{N}\mapsto \mathbb{N}$ such that any event structure of degree $\le n$ admits a nice labeling with at most $f(n)$ labels.}

Assous et al. \cite{AsBouChaRo}
proved that the event structures of degree 2 admit   nice labelings with 2 labels and noticed that Dilworth's theorem implies that the conflict-free event structures
of degree $n$ have nice labelings with $n$ labels. They also proved that finding the least number of labels in a nice labeling of a finite
event structure is NP-hard (by a reduction from graph coloring problem) and presented an example of a event structure of degree $n$ requiring more than $n$ labels. Recently,
Santocanale \cite{Sa} proved that all event structures of degree  3 and
with tree-like partial orders  have nice labelings with 3 labels. Both papers \cite{AsBouChaRo,Sa} contain some other results and reformulations of the nice labeling problem.
In particular,  Santocanale \cite{Sa} reformulated a nice labeling of an event structure $\mathcal E$ as a coloring
problem of the {\it orthogonality graph} ${\mathcal G}({\mathcal E})$ of $\mathcal E$: the vertices of ${\mathcal G}({\mathcal E})$ are the events of
$\mathcal E$ and two events are adjacent in
${\mathcal G}({\mathcal E})$ if and only if they are orthogonal (i.e., independent). Then the independent sets of events become the cliques of  ${\mathcal G}({\mathcal E})$,
the degree  of $\mathcal E$ becomes the  clique number of ${\mathcal G}({\mathcal E}),$ and the colorings of ${\mathcal G}({\mathcal E})$ are in bijection with the
nice labelings of $\mathcal E$.

In this note, we show that the conjecture of Rozoy and Thiagarajan is false already for event structures of degree 5. For this, we will use a more geometric and combinatorial
view on event structures. Namely, we will use the bijections between domains of event structures and median graphs established by
Barth\'elemy and Constantin \cite{BaCo} and between median graphs and CAT(0) cubical complexes established in \cite{Ch_CAT,Ro}. Together with those ingredients,  our counterexample
is based on  the Burling's construction \cite{Bu,Gy} of 3-dimensional box hypergraphs with clique number 2
and arbitrarily large chromatic numbers.

\section{Domains of event structures}

In view of our geometric approach to event structures, it will be more convenient to reformulate and investigate the nice labeling conjecture in terms of domains, which we recall now. The domain of an event structure $\mathcal E$ consists of all computations states, called configurations. Each computation state is a subset of events subject to the constrains that no two conflicting events can occur together in the same computation and
if an event occurred in a computation then all events on which it causally depends have occurred too. Formally, the set $\mathcal D=\mathcal D({\mathcal E})$ of {\it configurations} of an event structure ${\mathcal E}=(E,\le,\smile)$ consists of those subsets $C\subseteq E$ which are {\it conflict-free} ($e,e'\in C$ implies that $e,e'$ are not in conflict) and {\it downward-closed} ($e\in C$ and $e'\le e$ implies that $e'\in C$)  \cite{WiNi}. The {\it domain} of an event structure is the set $\mathcal D({\mathcal E})$ ordered by inclusion: if $C,C'\in {\mathcal D}({\mathcal E})$ and $C'\subseteq C,$ then $C'$ can be viewed as a subbehaviour of $C$. Thus the partial order $\subseteq$ on $\mathcal D({\mathcal E})$ expresses the progress in computation \cite{WiNi}.
As is noticed in \cite{Sa}, $(C',C)$ is a (directed) edge of the Hasse diagram of $({\mathcal D}({\mathcal E}),\subseteq)$ if and only if $C=C'\cup \{ e\}$ for an event $e\in E\setminus C,$ i.e., citing \cite{WiNi},
``events manifest themselves as atomic jumps from one configuration to another''. Then a nice labeling of the event structure ${\mathcal E}$ can be reformulated as a coloring of the directed edges of the Hasse diagram of its
domain ${\mathcal D}({\mathcal E})$ subject to the following local conditions \cite{Sa}:

\medskip\noindent
{\bf Determinism:} transitions outgoing from the same state have different colors, i.e., the  edges outgoing from the same vertex of ${\mathcal D}({\mathcal E})$ have different colors;

\noindent
{\bf Concurrency:} the opposite edges of each square of the Hasse diagram of ${\mathcal D}({\mathcal E})$ are colored in the same color.

\medskip
 As noticed in \cite{Sa}, there exists a bijection between such  edge-colorings of the Hasse diagram of ${\mathcal D}({\mathcal E})$ and nice labelings of $\mathcal E$ (i.e., colorings of the orthogonality graph  ${\mathcal G}({\mathcal E})$ of  $\mathcal E$). Moreover, it is shown in Lemma 2.11 of \cite{Sa} that $\{ e_1,\ldots,e_n\}$ is a clique of  ${\mathcal G}({\mathcal E})$ if and only if there exists a configuration $C\in {\mathcal D}({\mathcal E})$ such that for all $i=1,\ldots,n,$ $C\cup\{ e_i\}$ are configurations and $(C,C\cup\{ e_i\})$ are directed edges of the Hasse diagram of
${\mathcal D}({\mathcal E})$. According to this result, the degree of an event structure $\mathcal E$ (alias the clique-number of the orthogonality graph of $\mathcal E$)
equals to the maximum out-degree of a vertex in the Hasse diagram of  ${\mathcal D}({\mathcal E}).$

\section{Domains, median graphs, and CAT(0) cubical complexes}

We recall now the bijections between domains of event structures
and  median graphs established in \cite{BaCo} and between median
graphs and 1-skeletons of CAT(0) cubical
complexes established in \cite{Ch_CAT,Ro}. This will allow us to reformulate the nice labeling problem in truly geometric terms.
Median graphs and
related median structures (median algebras and
CAT(0) cubical complexes) have many nice properties and admit numerous
characterizations. These structures have been investigated in
several  contexts by quite a number of authors for more than half
a century.  We present here only a brief account of the
characteristic properties of median structures; for more detailed
information, the interested reader can consult the surveys
\cite{BaCh_survey, BaHe} and the book \cite{vdV} (see also the papers \cite{AbGh,GhPe} in which
CAT(0) cubical complexes are viewed as state complexes associated to metamorphic robots).

Let $G=(V,E)$ be simple, connected,
without loops or multiple edges, but not necessarily finite graph. The {\it distance}
$d(u,v)$ between two vertices $u$ and $v$ is the length of
a shortest $(u,v)$-path, and the {\it interval} $I(u,v)$ between $u$
and $v$ consists of all vertices on shortest $(u,v)$--paths, that
is, of all vertices (metrically) {\it between} $u$ and $v$:
$$I(u,v):=\{ x\in V: d(u,x)+d(x,v)=d(u,v)\}.$$
An induced subgraph of $G$ (or the corresponding vertex set)
is called {\it convex} if it includes the interval of $G$ between
any of its vertices.   A graph
$G=(V,E)$ is {\it isometrically embeddable} into a graph $H=(W,F)$
if there exists a mapping $\varphi : V\rightarrow W$ such that
$d_H(\varphi (u),\varphi (v))=d_G(u,v)$ for all vertices $u,v\in V$.

A graph $G$ is called {\it median} if the interval intersection
$I(x,y)\cap I(y,z)\cap I(z,x)$ is a singleton for each triplet $x,y,z$ of vertices. Median graphs are bipartite.
Basic examples of median graphs are trees (which are successive point
amalgams of $K_2$), hypercubes (which are Cartesian powers of
$K_2$), rectangular grids (which are Cartesian products of finite or infinite paths), and
covering graphs of distributive lattices.  With any vertex $v$ of a median
graph $G=(V,E)$ is associated a canonical partial order $\le_v$ defined by setting $x\le_v y$
if and only if $x\in I(v,y);$ $v$ is called the basepoint of $\le_v$. Since $G$ is bipartite,
the Hasse diagram $G_v$ of the partial order $(V,\le_v)$ is the graph $G$ in which any edge
$xy$ is directed from $x$ to $y$ if and only if the inequality $d(x,v)<d(y,v)$ holds. We call
$G_v$ a {\it pointed median graph}.  Theorems 2.2 and 2.3 of Barth\'elemy and Constantin \cite{BaCo} establish
the following bijection between event structures and pointed median graphs
(in \cite{BaCo},  event structures are called sites):

\begin{theorem}\cite{BaCo} \label{median_domain} The (undirected) covering graph
of the domain $({\mathcal D}({\mathcal E}),\subseteq)$ of  any event structure ${\mathcal E}=(E,\le, \smile)$ is a
median graph. Conversely, for any median graph $G$ and
any basepoint $v$ of $G,$  the pointed median graph $G_v$ is isomorphic to the Hasse diagram of
a domain of an event structure.
\end{theorem}

We only recall how to define the event structure occurring in the second part of this theorem. For this,
we will introduce  some notions which will be also used in the description of our counterexample.
Median graphs are isometric subgraphs of
hypercubes and Cartesian product of trees \cite{BaVdV,Mu}.  The isometric embedding of a median graph $G$ into a (smallest)
hypercube coincides with the so-called canonical embedding, which is
determined by the Djokovi\'c-Winkler relation $\Theta$ on the edge
set of $G:$ two edges $xy$ and $zw$ are $\Theta$-related exactly
when
$$d_G(x,z)+d_G(y,w)\ne d_G(x,w)+d_G(y,z).$$ For a
median graph this relation is transitive and hence an equivalence
relation.  It is the transitive closure of the ``opposite" relation
of edges on 4-cycles: in fact, any two $\Theta$-related edges can be
connected by a ladder (viz., the Cartesian product of a path with
$K_2$), and the block of all edges $\Theta$-related to some edge
$xy$ constitute a cutset $\Theta(xy)$ of the median graph, which
determines one factor of the canonical hypercube \cite{Mu1,Mu}. The
cutset $\Theta(xy)$ defines a convex split $\sigma(xy)=\{
W(x,y),W(y,x)\}$ of $G$ \cite{Mu}, where $W(x,y)=\{ z\in V:
d(z,x)<d(z,y)\}$ and $W(y,x)=V-W(x,y)$ (we will call the complementary convex sets $W(x,y)$
and $W(y,x)$ {\it halfspaces}; they are not only convex but also gated, see the
definition below). Conversely, for every convex
split of a median graph $G$ there exists at least one edge $xy$ such
that $\{ W(x,y),W(y,x)\}$ is the given split. We will denote by
$\{ \Theta_i: i\in I\}$ the equivalence classes of the relation
$\Theta$ (in \cite{BaCo},
they were called parallelism classes).

Suppose that $v$ is an arbitrary but fixed basepoint  of a median graph $G.$
For an equivalence class $\Theta_i, i\in I,$ we will denote by
$\sigma_i=\{ A_i,B_i\}$ the associated convex split, and suppose without loss of generality that  $v\in A_i.$  Two equivalence classes $\Theta_i$ and $\Theta_j$
are said to be {\it incompatible} or {\it crossing} if there exists a 4-cycle $C$ of $G$ with two opposite
edges in $\Theta_i$ and two other opposite edges in $\Theta_j$ ($\Theta_i$ and $\Theta_j$  are called \emph{compatible} otherwise).
An equivalence class $\Theta_i$  {\it separates} the basepoint $v$ from the
equivalence class $\Theta_j$ if
$\Theta_i$ and $\Theta_j$ are compatible and all edges of $\Theta_j$ belong to $B_i.$ The event structure ${\mathcal E}_v=(E,\le, \smile)$ associated with a pointed median graph $G_v$ is defined in the following way. $E$ is the set
$\{ \Theta_i, i\in I\}$ of the equivalence classes of $\Theta$. The causal dependency is defined by setting $\Theta_i\le \Theta_j$ if and only if $\Theta_i=\Theta_j$
or $\Theta_i$ separates $v$ from $\Theta_j$. Finally, the conflict relation is defined by setting $\Theta_i\smile\Theta_j$ if and only if $\Theta_i$ and $\Theta_j$
are compatible, $\Theta_i$ does not separate $v$ from $\Theta_j$ and $\Theta_j$ does not separates $v$ from $\Theta_i.$ Theorem 2.3 of \cite{BaCo} shows that $G_v$ is indeed
the Hasse diagram of the domain ${\mathcal D}({\mathcal E}_v)$ of the event structure ${\mathcal E}_v$.

In our counterexample we will use the following constructive characterization of finite median graphs.
A subset $W$ of $V$
or the subgraph $H$ of $G=(V,E)$ induced by $W$ is called
{\it gated} (in
$G$) if for every vertex $x$ outside $H$ there exists a vertex $x'$
(the {\it gate} of $x$) in $H$ such that each vertex $y$ of $H$ is
connected with $x$ by a shortest path passing through the gate $x'.$ In general, any gated set is convex.
In median graphs, all convex sets are gated.
A graph $G$ is a {\it gated amalgam} of two
graphs $G_1$ and $G_2$ if $G_1$ and $G_2$ constitute  two
intersecting gated subgraphs of $G$ whose union is all of $G.$ Equivalently, $G$ is a gated amalgam of
$G_1$ and $G_2$  if the intersection $G_0$ of $G_1$ and $G_2$ in $G$ is a gated subgraph of $G_1$ and $G_2$ and $G_0$
separates in $G$
any vertex of $G_1\setminus G_2$ from any vertex of $G_2\setminus G_1.$

\begin{theorem} \cite{Is,vdV1} \label{amalgam} The gated amalgam of two median graphs is a median graph.
Moreover, every finite median graph $G$ can be obtained by
successive applications of gated amalgamations from hypercubes.
\end{theorem}

Now we recall the close relationship between the median graphs and CAT(0)
cubical complexes.  We believe that it is worth putting together event structures, median graphs, and CAT(0)
cubical complexes because some problems similar to the nice labeling problem
have been independently formulated in geometric and combinatorial  settings.
Notice also that a result similar to the result of Barth\'elemy and Constantin \cite{BaCo}
has been rediscovered recently in \cite{ArOwSu} in the context of CAT(0) cubical complexes.
Finally, it may happen that combinatorial, structural, algebraic, geometrical, and group theoretical
results established for median structures or CAT(0) cubical complexes can be useful for the
investigation of  event structures.

A {\it cubical complex} $\mathcal K$ is a  set of solid cubes of any
dimensions which is closed under taking subcubes and nonempty
intersections.  For a complex $\mathcal K$ denote by
$V({\mathcal K})$ and $E({\mathcal K})$ the {\it vertex set} and the
{\it edge set} of ${\mathcal K},$ namely, the set of all
0-dimensional and 1-dimensional cubes of ${\mathcal K}.$ The pair
$G({\mathcal K})=(V({\mathcal K}),E({\mathcal K}))$ is called the {\it (underlying)
graph} or the {\it 1-skeleton} of ${\mathcal K}.$ Conversely, for a graph $G$ one can derive  a cubical complex
${\mathcal K}(G)$ by replacing all graphic cubes of $G$ by solid cubes.
The cubical complex ${\mathcal K}(G)$ associated with a median
graph $G$ is called a {\it median cubical complex}.  If instead of a
solid cube, we replace each graphic cube of $G$ by an
axis-parallel box (i.e., if instead of length 1 we
take length $l_i$ for all edges from the same equivalence
class $\Theta_i$), then we will get a {\it median box complex}.  Median
cubical and median box
complexes endowed with the intrinsic $l_1$-metric are median metric
spaces  (i.e., every triplet of points has a unique median) and therefore are
$l_1$-subspaces \cite{vdV}.  Finally, if we impose the intrinsic
$l_2$-metric on a median cubical or box complex, then we obtain a metric space
with global non-positive curvature.

A {\it geodesic triangle} $\Delta=\Delta (x_1,x_2,x_3)$ in a geodesic
metric space $(X,d)$ consists of three points in $X$ (the vertices
of $\Delta$) and a geodesic  between each pair of vertices (the
sides of $\Delta$). A {\it comparison triangle} for $\Delta
(x_1,x_2,x_3)$ is a triangle $\Delta (x'_1,x'_2,x'_3)$ in the
Euclidean plane  ${\mathbb E}^2$ such that $d_{{\mathbb
E}^2}(x'_i,x'_j)=d(x_i,x_j)$ for $i,j\in \{ 1,2,3\}.$ A geodesic
metric space $(X,d)$ is defined to be a {\it CAT(0) space}
\cite{Gr} if all geodesic triangles $\Delta (x_1,x_2,x_3)$ of $X$
satisfy the comparison axiom of Cartan--Alexandrov--Toponogov:
{\it If $y$ is a point on the side of $\Delta(x_1,x_2,x_3)$ with
vertices $x_1$ and $x_2$ and $y'$ is the unique point on the line
segment $[x'_1,x'_2]$ of the comparison triangle
$\Delta(x'_1,x'_2,x'_3)$ such that $d_{{\mathbb E}^2}(x'_i,y')=
d(x_i,y)$ for $i=1,2,$ then $d(x_3,y)\le d_{{\mathbb
E}^2}(x'_3,y').$} CAT(0) spaces can be characterized  in several different  natural ways, in particular,
a geodesic metric space $(X,d)$ is CAT(0) if and only if
any two points of this space can be joined by a unique geodesic.
Several classes of
CAT(0) complexes can be characterized
combinatorially, and the characterization of cubical CAT(0)
complexes given  by M. Gromov is especially nice:

\begin{theorem}  \cite{Gr} A cubical or box complex
${\mathcal K}$ with the $l_2$-metric is CAT(0) if and
only if ${\mathcal K}$ is simply connected and
whenever three $(k + 2)$-cubes of ${\mathcal
K}$ share a common $k$-cube and pairwise share common $(k
+1)$-cubes, they are contained in a $(k+3)$--cube of ${\mathcal
K}.$
\end{theorem}

The following relationship holds between
CAT(0) cubical complexes and median cubical complexes.

\begin{theorem} \label{cat0} \cite{Ch_CAT,Ro} Median cubical (box) complexes  and CAT(0) cubical (box) complexes
(both equipped with the $l_2$-metric) constitute the same
objects.
\end{theorem}

The proof of this theorem given in \cite{Ch_CAT} is self-contained and allows to derive some properties of
CAT(0) cubical complexes from known results about median graphs. In particular, a fundamental result of Sageev \cite{Sa}
that each hyperplane of a CAT(0) cubical complex $\mathcal K$ does not self-intersect and partition $\mathcal K$  in exactly
two parts is a consequence of the fact that each equivalence class $\Theta_i$ of the median graph $G({\mathcal K})$ defines a convex split $\{ A_i,B_i\}.$
A {\it hyperplane} $H_i$ associated to $\Theta_i$ is the cubical complex whose 1-skeleton is the graph in which the middles (baricenters)  of the
edges of $\Theta_i$ are the vertices and
two such vertices are adjacent if they are middles of two opposite edges of a square of $\mathcal K$.  The {\it carrier} $N(H_i)$ in $\mathcal K$
of a hyperplane $H_i$  is the union of all cubes of $\mathcal K$ crossed by $H_i,$ i.e. the union of all cubes having an
edge in the equivalence class $\Theta_i$. In \cite{Ha}, Hagen introduced and investigated in depth the important concept of a contact
graph of a CAT(0) cubical complex $\mathcal K$. According to \cite{Ha}, the contact graph $\Gamma=\Gamma({\mathcal K})=\Gamma(G({\mathcal K}))$ of
$\mathcal K$ is a graph having the  hyperplanes (or the  equivalence classes of $\Theta$)
as vertices and two hyperplanes $H_i$ and $H_j$ are
adjacent in $\Gamma$ (notation $H_i\coll H_j$) if and only if the carriers $N(H_i)$ and $N(H_j)$ intersect. It was noticed in \cite{Ha}
that if $H_i\coll H_j,$ then the hyperplanes $H_i$ and $H_j$ either {\it cross} (in which case the equivalence classes $\Theta_i$ and $\Theta_j$ cross)
or {\it osculate} (in which case there exist two edges $e\in \Theta_i$ and $e'\in \Theta_j$ sharing a common endpoint and not belonging to a common square). Analogously to the equality between the
clique-number of the orthogonality graph of an event structure and the degree of the event structure, the clique number $\omega(\Gamma)$ of the
contact graph equals to the maximum degree of the graph $G({\mathcal K})$ of $\mathcal K$.

Due to the established bijections between  event structures and their domains, between domains
and pointed median graphs, and between  median graphs and CAT(0) cubical complexes, one can view the orthogonality graph of an event
structure and the crossing graph of a median graph as subgraphs of the contact graph of the associated CAT(0) cubical complex.
Namely, the {\it crossing graph} $\Gamma_{\#}=\Gamma_{\#}(G)=\Gamma_{\#}({\mathcal K}(G))$
of a median graph $G$ (or of the associated cubical complex ${\mathcal K}(G)$) has the hyperplanes of
${\mathcal K}(G)$ (or the equivalence classes of $\Theta$) as vertices and the pairs of crossing hyperplanes as edges.
Now, let $G_v$ be a pointed median graph obtained from $G$. As we noticed already (and this follows easily from the definition of the halfspaces), all edges of any equivalence class $\Theta_i$ of $G$ are directed in $G_v$ from $A_i$ to $B_i,$ i.e., if $xy\in \Theta_i$ and $v,x\in A_i,y\in B_i,$
then $xy$ is directed from $x$ to $y.$ The {\it pointed  contact graph} $\Gamma_v=\Gamma_v(G)=\Gamma_{v}({\mathcal K}(G))$ of $G$ has the set of hyperplanes of ${\mathcal K}(G)$ (or the equivalence classes of $\Theta$) as vertices and two hyperplanes $H_i$ and $H_j$ are adjacent if and only if either they cross or they osculate in two directed edges $e\in \Theta_i$ and $e'\in \Theta_j$ with a common origin (one can view $\Gamma_v$ as the orthogonality graph of the event structure having the pointed median graph $G_v$ as a domain). Since the relation $\Theta$ is transitive, in any coloring of the edges of the pointed median graph $G_v$ satisfying the  determinism and concurrency conditions, all edges of an equivalence class $\Theta_i$  have the same color,  two crossing equivalence classes have different colors, and two edges with common origin have different colors. Hence, the colorings of edges of $G_v$ are in bijection with the colorings of the pointed contact graph $\Gamma_v$.  On the other hand, if we color the equivalence classes of $\Theta$ so that this is a coloring of the crossing graph $\Gamma_{\#}$ of $G,$ then this
corresponds to a coloring of edges of $G$ using the concurrency rule only: the edges of each square of $G$ are colored in two colors  with opposite edges having the same
color. By Proposition 1 of \cite{BaChEp_ramified}, coloring  $\Gamma_{\#}$ in $n$ colors is equivalent to an isometric embedding of $G$ into the
Cartesian product of $n$ trees.

One can ask whether the {\it chromatic number of each of the graphs $\Gamma, \Gamma_{\#},$ or $\Gamma_v$ is bounded by a function of its clique number}. In case of the pointed contact graph $\Gamma_v$, this is exactly the nice labeling problem. In case of the contact graph $\Gamma$, this question was raised by Hagen in the first version of \cite{Ha}.  On the other
hand, it is well-known \cite{BaVdV} (see also Proposition 2.17 of \cite{Ha})  that any graph can be realized as the crossing graph of a median graph. Since there exists triangle-free graphs with arbitrarily high chromatic numbers, the chromatic number of $\Gamma_{\#}$ cannot be bounded by its clique number. Nevertheless, M. Sageev (personal communication from M. Hagen) and, independently, the author of the present note asked whether the {\it chromatic number of $\Gamma_{\#}=\Gamma_{\#}(G)$ is bounded by a function of the maximum degree of the median graph $G$, i.e., if a median graph $G$ with bounded degrees of vertices can be isometrically embedded into a bounded number of trees.}. In the next section, we will answer in the negative the first question about the graphs $\Gamma$ and $\Gamma_v$. In a forthcoming paper with Hagen, we will modify this example to answer in the negative the last question about $\Gamma_{\#}.$

\section{The counterexample}

Our counterexample to the labeling conjecture of \cite{RoTh} is based on examples of Burling of box hypergraphs with clique number 2 and arbitrarily  large chromatic numbers.
Parallelepipeds in ${\mathbb R}^3$ whose sides are parallel to the coordinate axes are called {\it 3-dimensional boxes}. Given a (finite) collection ${\mathcal B}$
of 3-dimensional boxes, its clique number $\omega({\mathcal B})$ is the maximum number of pairwise intersecting boxes of
$\mathcal B$ and its chromatic number $\chi({\mathcal B})$ is the minimum number of colors into which we can color the boxes of $\mathcal B$ in such a way that any
pair of intersecting boxes is colored in different colors (i.e., $\chi({\mathcal B})$ is the chromatic number of the intersection graph of the family $\mathcal B$).
In his PhD thesis \cite{Bu}, for each integer $n>0$, Burling constructed a collection of axis-parallel boxes ${\mathcal B}(n)$ with clique number $\omega({\mathcal B}(n))=2$
and chromatic number $\chi({\mathcal B}(n))>n$ (a full description of this construction is available in the survey paper by Gy\'arf\'as \cite{Gy}).

Let $B_0$ be a box of ${\mathbb R}^3.$ Suppose without loss of generality that one corner of $B_0$ is the origin of coordinates of ${\mathbb R}^3$ and that
$B_0$ is located in the first octant of ${\mathbb R}^3.$ Suppose that $B_0$ is subdivided into smaller boxes (called {\it elementary cells}) using a family of planes parallel to the three coordinate hyperplanes. This subdivision of $B_0$ defines a box complex as well as a cubical complex (if we scale all length of edges of the resulting boxes to 1). We
denote both these complexes by ${\mathcal K}$  and by $G=G({\mathcal K})$ their 1-skeleton. Notice that if $B_0$ is subdivided by $k_1-2$ planes parallel to the $xy$-plane,
$k_2-2$ planes parallel to the $yz$-plane, and $k_3-2$ planes parallel to the $xz$-plane, then $G$ is isomorphic to the  $k_1\times k_2\times k_3$ grid, and therefore is a median graph ($\mathcal K$ is a CAT(0) cubical complex because its underlying space is the box $B_0$).

Let ${\mathcal B}=\{ B_1,\ldots,B_m\}$ be a box hypergraph  such that the eight corners of each box $B_i, i=1,\ldots,m,$ are vertices of the grid $G.$ In this case, we say that the box hypergraph $\mathcal B$ is {\it cell-represented} by the box complex ${\mathcal K}$ because  each box
$B_i$ is the union of elementary cells of $\mathcal K$.  Let ${\mathcal K}_i$
be the subcomplex of $\mathcal K$ consisting of all elementary cells included in $B_i$ and let $G_i=G({\mathcal K}_i)$ be its underlying graph. Note that $G_i$ is also a 3-dimensional grid.
Hence $G_i$ is a convex  (and therefore gated) subgraph of  $G$.

Now, we define a {\it lifting procedure} taking as an input the box complex $\mathcal K$, its graph $G$, and a box hypergraph ${\mathcal B}$ cell-represented by $\mathcal K,$ and giving rise to a 4-dimensional
CAT(0) box  complex $\widetilde{\mathcal K}$ and its underlying median graph $\widetilde{G}=G(\widetilde{\mathcal K}).$ The complex $\widetilde{\mathcal K}$ is realized in the $(m+3)$-dimensional space ${\mathbb R}^{m+3}$. Suppose that the 3-dimensional space in which we defined  the box complex $\mathcal K$ is the subspace of ${\mathbb R}^{m+3}$ defined by the last 3 coordinates, i.e., each point $p$ of $B_0$ has the coordinates $(0,\ldots,0,p_{m+1},p_{m+2},p_{m+3})$ with $p_{m+1},p_{m+2},p_{m+3}\ge 0.$ Then for each box $B_i, i=1,\ldots,m,$ we can define the numbers $0\le a'_i<a''_i,0\le b'_i<b''_i, 0\le c'_i<c''_i$ so that $B_i$ is the set of all points $p=(0,\ldots,0,p_{m+1},p_{m+2},p_{m+3})\in {\mathbb R}^{m+3}$ such that $p_{m+1}\in [a'_i,a''_i],$ $p_{m+2}\in [b'_i,b''_i],$ and $p_{m+3}\in [c'_i,c''_i].$ Let  $\widetilde{B}_i$ be the 4-dimensional box which is the Cartesian product of $B_i$ with the unit segment $s_i$ of the $i$th coordinate-axis of ${\mathbb R}^{m+3}$:   $\widetilde{B}_i$ consists of all points $p=(p_1,\ldots,p_m,p_{m+1},p_{m+2},p_{m+3})\in {\mathbb R}^{m+3}$ such that $p_j=0$ for each $1\le j\le m, j\ne i,$ $p_i\in [0,1],$ $p_{m+1}\in [a'_i,a''_i], p_{m+2}\in [b'_i,b''_i],$ and $p_{m+3}\in [c'_i,c''_i].$ Let $\widetilde{\mathcal B}=\{ \widetilde{B}_i: B_i\in {\mathcal B}\}$ be the resulting box hypergraph in ${\mathbb R}^{m+3}$. Each elementary cell $C$ of $\mathcal K$ gives rise to a 4-dimensional box $\widetilde{C}_i=C\times s_i$ for each 3-dimensional box $B_i\in {\mathcal B}$ containing $C.$ We refer to each $\widetilde{C}_i$ as a {\it lifted elementary box}. Denote by $\widetilde{\mathcal K}$ the box complex consisting of all elementary cells of $\mathcal K$ and of all lifted elementary cells. For each box $\widetilde{B}_i$, let   $\widetilde{{\mathcal K}}_i$ be the subcomplex of $\widetilde{\mathcal K}$ consisting of all lifted elementary cells $\widetilde{C}_i$ such that $C\subseteq B_i.$  Finally, let $\widetilde{G}=G(\widetilde{\mathcal K})$ and $\widetilde{G}_i=G(\widetilde{\mathcal K}_i)$ be the 1-skeletons of $\widetilde{\mathcal K}$  and $\widetilde{{\mathcal K}}_i,$ respectively. Notice also that each
$\widetilde{G}_i$ is a 4-dimensional rectangular grid because $\widetilde{G}_i$ is the Cartesian product $\widetilde{G}_i=G_i\times e_i$ of the grid  $G_i$ with an edge $e_i.$   Notice that the edges $e_1,\ldots,e_m$ are different because the unit segments $s_1,\ldots,s_m$ belong to different coordinate-axes of ${\mathbb R}^{m+3}.$

\begin{lemma} \label{median} $\widetilde{G}$ is a median graph.
\end{lemma}

\begin{proof} We will show that $\widetilde{G}$ is obtained from the median graphs $G,\widetilde{G}_1,\ldots,\widetilde{G}_{m-1},\widetilde{G}_m$ by applying successive gated amalgamations.  We proceed by induction on the size of the box hypergraph ${\mathcal B}.$ If ${\mathcal B}=\varnothing,$ then the result is immediate because $\widetilde{G}$ coincide with $G.$ Now suppose by induction assumption that our assertion holds for the box hypergraph ${\mathcal B}'=\{ B_1,\ldots,B_{m-1}\}.$ Let $\widetilde{G}'$ be the median graph which is the 1-skeleton of the box complex $\widetilde{\mathcal K}'$ defined for the box hypergraph $\widetilde{{\mathcal B}}'=\{ \widetilde{B}_1,\ldots, \widetilde{B}_{m-1}\}.$ Notice that $\widetilde{G}$ is obtained by amalgamating the median graphs $\widetilde{G}'$ and $\widetilde{G_i}$ along their common subgraph $G_i.$ Since $G_i$ is a grid, $G_i$  is a gated subgraph of the grid $\widetilde{G}_i$. Analogously,  since $G_i$ is a gated subgraph of $G$
and the grid $G$ is a gated subgraph of each of the median graphs resulting from previous gated amalgams, $G_i$ is a gated subgraph of $\widetilde{G}'$. Thus $\widetilde{G}$ is a gated amalgam of two median graphs $\widetilde{G}_i$ and $\widetilde{G}',$ whence $\widetilde{G}$ is a median graph by Theorem \ref{amalgam}.
\end{proof}

Let $\Theta_i, i=1,\ldots,m,$ be the equivalence class of the edges of the median graph $\widetilde{G}$ defined by $e_i,$ i.e., $\Theta_i$ consists of all edges of $\widetilde{G}$ (or of $\widetilde{\mathcal K}$) which are parallel to the $i$th coordinate-axis of ${\mathbb R}^{m+3}.$   Each edge of $\Theta_i$ has one endpoint in the grid $G_i$ of the box $B_i$ and, vice versa, each vertex of $G_i$ is an endpoint of exactly one edge of $\Theta_i$. $\Theta$ comprises also other equivalence classes, namely the equivalence classes of the grid $G$ augmented by the edges of $\widetilde{G}_1,\ldots,\widetilde{G}_m$ parallel to them.

Let $\alpha$ be the corner of $B_0$ which is identified with the origin of coordinates of ${\mathbb R}^{m+3}$ and let $\widetilde{G}_{\alpha}$ be the median graph $\widetilde{G}$ pointed at $\alpha$. Then each edge $e$ of the equivalence class $\Theta_i$ will be directed in $\widetilde{G}_{\alpha}$ away from $G_i,$ i.e., the endpoint of $e$ from $G_i$ will be the origin of $e$. The edges $e=uv$ of the remaining equivalence classes will be directed in $\widetilde{G}_{\alpha}$ from a vertex with smaller coordinates to a vertex with larger coordinates (notice that, if fact, $u$ and $v$ will differ in exactly one of the last three  coordinates).  Let $\Gamma_{\alpha}(\widetilde{G})$ be the pointed contact graph of $\widetilde{G}_{\alpha}$.

\begin{lemma} \label{crossing} Any two equivalence classes $\Theta_i$ and $\Theta_j$ do not cross.
\end{lemma}

\begin{proof} The hyperplane $H_i$ of the CAT(0) box complex $\widetilde{\mathcal K}$ defined by $\Theta_i$ lies in the (3-dimensional) plane $\Pi_i$ of ${\mathbb R}^{m+3}$ described by the equations $x_k=0$ if $1\le k\le m, k\ne i,$ and $x_i=\frac{1}{2}.$ Analogously, the hyperplane $H_j$ defined by $\Theta_j$ lies in the plane $\Pi_j$ of ${\mathbb R}^{m+3}$ described by the equations $x_k=0$ if $1\le k\le m, k\ne j,$ and $x_j=\frac{1}{2}.$ Since $\Pi_i$ and $\Pi_j$ are disjoint, the hyperplanes $H_i$ and $H_j$ are disjoint as well, thus the equivalence classes $\Theta_i$ and $\Theta_j$ do not cross.
\end{proof}

\begin{lemma} \label{intersection} Two equivalence classes $\Theta_i$ and $\Theta_j$ are adjacent in  $\Gamma_{\alpha}(\widetilde{G})$ if and only if the grids $G_i$ and $G_j$ intersect and if and only if the boxes $B_i$ and $B_j$ intersect.
\end{lemma}

\begin{proof} First suppose that $\Theta_i$ and $\Theta_j$ are adjacent in  $\Gamma_{\alpha}(\widetilde{G})$. By Lemma \ref{crossing} and the definition of edges of $\Gamma_{\alpha}(\widetilde{G}),$ there exist two directed edges $e'\in \Theta_i$ and $e''\in \Theta_j$ having the same origin. Since the origin of $e'$ belongs to $G_i$ and the origin of $e''$ belongs to $G_j,$ we conclude that this common origin belongs  to $G_i\cap G_j$ and therefore to $B_i\cap B_j.$

Conversely, suppose that $B_i$ and $B_j$ intersect. Since each of $B_i$ and $B_j$ is constituted by elementary cells of ${\mathcal K}$, we conclude that there exist two elementary cells $C'\subseteq B_i$ and $C''\subseteq B_j$ which intersect. Necessarily $C'$ and $C''$ share a common vertex $v$ of $G$. Hence $v$ is a vertex of both grids $G_i$ and  $G_j$. Now, from our construction follows that $v$ is the origin of an edge $e'$ of $\Theta_i$ and of an edge $e''$ of $\Theta_j,$ whence the equivalence classes   $\Theta_i$ and $\Theta_j$ are adjacent in  $\Gamma_{\alpha}(\widetilde{G})$.
\end{proof}

\begin{lemma} \label{degree} If the clique number of the box hypergraph ${\mathcal B}$ is $\omega=\omega({\mathcal B})$, then the out-degree of a vertex in the pointed median graph
$\widetilde{G}_{\alpha}$ is $\omega+3.$ In particular, the maximum degree of $\widetilde{G}$ is at most $\omega+6$ and the clique number of the pointed contact
graph $\Gamma_{\alpha}(\widetilde{G})$ is $\omega+3.$
\end{lemma}

\begin{proof}  Consider a maximal by inclusion collection ${\mathcal B}_0$ of $k$ pairwise intersecting boxes of $\mathcal B$. By Lemma \ref{intersection}, for any two boxes $B_i,B_j\in {\mathcal B}_0,$ the grids $G_i$ and $G_j$ intersect. From Helly property for convex sets of median graphs, we conclude that all grids $G_i$ with $B_i\in {\mathcal B}_0$ share  a common vertex $v.$ The out-degree of $v$ in the pointed median graph $\widetilde{G}_{\alpha}$ is at most $k+3$ because $v$ is the origin of one edge from each of the $k$ equivalence classes $\Theta_i$ with $B_i\in {\mathcal B}_0$ as well as the origin of at most three outgoing edges in the pointed grid $G_{\alpha}$. Since $k\le \omega,$ the out-degree of $v$ is at most $\omega+3$ showing that the out-degree in $\widetilde{G}_{\alpha}$ of any vertex of the grid $G$ is at most $\omega+3.$ On the other hand, the out-degree of each vertex $z$ of $\widetilde{G}$ not belonging to $G$ is equal to the out-degree in $G_{\alpha}$  of its twin in $G$ and therefore is at most 3. Since any vertex of $\widetilde{G}$ can have at most three incoming edges, we conclude that the maximum degree of a vertex of the undirected graph  $\widetilde{G}$ is at most $\omega+6.$
\end{proof}

\begin{lemma} \label{colouring} $\chi(\Gamma(\widetilde{G}))\ge \chi(\Gamma_{\alpha}(\widetilde{G}))\ge \chi({\mathcal B}).$
\end{lemma}

\begin{proof} The first inequality is obvious because $\Gamma_{\alpha}(\widetilde{G})$ is a subgraph of $\Gamma(\widetilde{G})$. By Lemma \ref{intersection}, any coloring of $\Gamma(\widetilde{G})$
restricted to the equivalence classes $\Theta_i, i=1,\ldots,m,$ provides a coloring of the box hypergraph $\mathcal B$.
\end{proof}

Now we will apply our construction to Burling's examples. For each integer $n>0$, let  ${\mathcal B}(n)$  be a box hypergraph with clique number 2 and chromatic number $\chi({\mathcal B}(n))>n$ as defined in \cite{Bu,Gy}.
Suppose that  ${\mathcal B}(n)$ is drawn in the first open octant of ${\mathbb R}^3$.   Let $B_0(n)$ be an additional axis-parallel box having one corner in the origin $\alpha(n)$ of coordinates of ${\mathbb R}^3$ and containing all boxes of ${\mathcal B}(n)$ ($B_0(n)$ will play the role of the box $B_0$). Let $\beta(n)$ be the corner of $B_0(n)$ opposite to $\alpha(n).$
Now, subdivide $B_0(n)$ into elementary cells by drawing the three axis-parallel planes trough each of eight  corners of each box $B_i$ of ${\mathcal B}(n).$ Denote the resulting grid by $G(n)$ and the resulting box complex (subdividing $B_0(n)$) by ${\mathcal K}(n).$ Then the box hypergraph ${\mathcal B}(n)$ is cell-represented by ${\mathcal K}(n).$ Denote by $\widetilde{\mathcal K}(n)$ the box complex obtained by applying our lifting procedure to ${\mathcal K}(n)$ and ${\mathcal B}(n).$ Let  $\widetilde{G}(n)=G(\widetilde{\mathcal K}(n))$  be the 1-skeleton of $\widetilde{\mathcal K}(n)$ and let $\widetilde{G}_{\alpha(n)}(n)$ be the median graph $\widetilde{G}(n)$ pointed at $\alpha(n)$. Since $\omega({\mathcal B}(n))=2$, from Lemma \ref{degree} we conclude that the maximum out-degree of a vertex in the pointed median graph $\widetilde{G}_{\alpha_n}(n)$ is at most 5. On the other hand, since $\chi({\mathcal B}(n))>n,$ from Lemma \ref{colouring} we conclude that the chromatic numbers of the contact graph of $\widetilde{G}(n)$ and of the pointed contact graph of $\widetilde{G}_{\alpha_n}(n)$ are larger than $n.$ Summarizing, we obtain the following conclusion:

\begin{proposition} \label{G(n)} For any $n>0$, there exist a median graph $\widetilde{G}(n)$ of maximum degree 8 and a pointed median graph $\widetilde{G}_{\alpha_n}(n)$ of maximum out-degree 5 such that any coloring of the contact graph of $\widetilde{G}(n)$ and of the pointed contact graph of $\widetilde{G}_{\alpha_n}(n)$ requires more than $n$ colors. In particular, any nice labeling of the event structure  ${\mathcal E}_{\alpha(n)}$ (of degree 5) whose domain is $\widetilde{G}_{\alpha_n}(n)$ requires more than $n$ labels.
\end{proposition}

To present a counterexample to the conjecture of Rozoy and Thiagarajan, we consider the following infinite median graph $\widetilde{G}^*$ whose blocks (2-connected components) are the graphs $\widetilde{G}(1),\widetilde{G}(2),\ldots$. Recall that each graph $\widetilde{G}(n)$ has two distinguished vertices $\alpha(n)$ and $\beta(n)$ which are opposite corners of the box $B_0(n).$ To construct $\widetilde{G}^*,$ for each $n>1,$ we identify the vertex $\beta(n-1)$ of $\widetilde{G}(n-1)$ with the vertex $\alpha(n)$ of  $\widetilde{G}(n)$ and obtain an infinite in one direction chain of blocks. Notice that the identified vertices $\beta(n-1)=\alpha(n)$ are exactly the articulation vertices of $\widetilde{G}^*.$ Obviously $\widetilde{G}^*$ is a median graph because each its block is median.  Now suppose that  $\widetilde{G}^*$ is pointed at the vertex $\alpha=\alpha(1)$ and let  $\widetilde{G}^*_{\alpha}$ be the resulting pointed median graph. Notice that each edge $e$ of  $\widetilde{G}^*$ is oriented in $\widetilde{G}^*_{\alpha}$ in the same way as in the orientation $\widetilde{G}_{\alpha_n}(n)$ of the unique block $\widetilde{G}(n)$ containing $e.$ On the other hand, the out-degree in  $\widetilde{G}^*_{\alpha}$ of each vertex $v$ belonging to a unique block $\widetilde{G}(n)$ is the same as in $\widetilde{G}_{\alpha_n}(n)$ while the out-degree of each articulation point is 3, whence the maximum out-degree  of $\widetilde{G}^*_{\alpha}$ is also 5. The pointed contact graph $\Gamma(\widetilde{G}^*_{\alpha})$ of $\widetilde{G}^*_{\alpha}$ is the disjoint union of the pointed contact graphs  $\Gamma(\widetilde{G}_{\alpha_n}(n))$. From Proposition \ref{G(n)} we conclude that the chromatic number of $\Gamma(\widetilde{G}^*_{\alpha})$ is infinite.  As a consequence, we established the  following main result of this note:

\begin{theorem} \label{theorem} There exists a pointed median graph $\widetilde{G}^*_{\alpha}$ of maximum out-degree 5 such that the chromatic number of its pointed contact graph $\Gamma(\widetilde{G}^*_{\alpha})$ is infinite. In particular, any nice labeling of the event structure  ${\mathcal E}_{\alpha}$ (of degree 5) whose domain is $\widetilde{G}^*_{\alpha}$ requires an infinite number of labels.
\end{theorem}

\section*{Acknowledgements}

I would like to acknowledge Remi Morin for introducing to me the nice labeling conjecture and useful discussions, Luigi Santocanale for several discussions on this conjecture,
and Mark Hagen for our exchanges and collaboration on  the 2-dimensional case. This  work was supported in part by the ANR grant BLAN GGAA.

\end{document}